\documentclass[11pt,a4paper]{amsart}
\usepackage[utf8]{inputenc}
\usepackage[english]{babel}
\usepackage{amsmath,amssymb,amsfonts,xcolor,epsf,mathrsfs,bbm,subfig,mathtools, verbatim}
\usepackage{enumerate}
\usepackage{graphicx}
\usepackage{listings}
\usepackage{csquotes}

\usepackage{ytableau}

\usepackage{wasysym}
\usepackage{verbatim}
\usepackage{fullpage}

\usepackage[style=numeric]{biblatex}
\usepackage{xfrac}
\usepackage{tikz}
\usepackage{stmaryrd}
\usetikzlibrary{matrix}

\newcommand{\K}{\mathbb{K}}
\newcommand{\N}{\mathbb{N}}

\newcommand{\R}{\mathbb{R}}
\newcommand{\C}{\mathbb{C}}

\newcommand{\X}{\underline{X}}
\renewcommand{\leq}{\leqslant}
\usepackage{faktor}

\usepackage{MnSymbol} 
\newtheorem{theorem}{Theorem}
\newtheorem{lemma}[theorem]{Lemma}
\newtheorem{remark}[theorem]{Remark}

\newtheorem{proposition}[theorem]{Proposition}
\newtheorem{definition}[theorem]{Definition}
\newtheorem{example}[theorem]{Example}

\DeclareMathOperator\id{id}
\DeclareMathOperator\im{im}

\DeclareMathOperator{\Tr}{Tr}

\DeclarePairedDelimiter\floor{\lfloor}{\rfloor}
\DeclareFontFamily{U}{mathx}{\hyphenchar\font45}
\DeclareFontShape{U}{mathx}{m}{n}{
      <5> <6> <7> <8> <9> <10>
      <10.95> <12> <14.4> <17.28> <20.74> <24.88>
      mathx10
      }{}
\DeclareSymbolFont{mathx}{U}{mathx}{m}{n}
\DeclareFontSubstitution{U}{mathx}{m}{n}
\DeclareMathAccent{\widecheck}{0}{mathx}{"71}

\begin{document}

\title
{
Constructively describing orbit spaces of finite groups by few inequalities
}
\author{Philippe Moustrou}
\address{Institut de Mathématiques de Toulouse, UMR 5219, UT2J, 31058 Toulouse, France}
\email{philippe.moustrou@math.univ-toulouse.fr}
\author{Cordian Riener}
\address{Department of Mathematics and Statistics, UiT - the Arctic University of Norway, 9037 Troms\o, Norway}
\email{cordian.riener@uit.no}
\author{Robin Schabert }
\address{Department of Mathematics and Statistics, UiT - the Arctic University of Norway, 9037 Troms\o, Norway}
\email{robin.schabert@uit.no}

\begin{abstract}
Let $G$ be a finite group acting linearly on $\R^n$. A celebrated Theorem of Procesi and Schwarz gives an explicit description of the orbit space $\R^n  /\!/G$ as a basic closed semi-algebraic set. We give a new proof of this statement and another description as a basic closed semi-algebraic set using elementary tools from real algebraic geometry. Bröcker was able to  show that the number of inequalities needed to describe the orbit space generically depends only on the group $G$. Here, we construct such inequalities explicitly for abelian groups and in the case where only one inequality is needed. Furthermore, we answer an open question raised by Bröcker concerning the genericity of his result.
\end{abstract}

\maketitle

\section{Introduction}

A set $S \subset \K^n$ defined as the intersection of finitely many polynomial inequalities is called a \emph{basic semi-algebraic set}. Sets obtained as finite unions or complements of such basic sets are known as \emph{semi-algebraic} sets. A fundamental statement in real algebraic geometry, attributed to Tarski and Seidenberg, asserts that this class of sets is closed under polynomial maps. However, obtaining an explicit description of the image for a given semi-algebraic set and a specific polynomial map is far from trivial. Furthermore, although the image of a semi-algebraic set under a polynomial map is also semi-algebraic, it is generally not true that the image of a basic semi-algebraic set remains basic.

In this article, we investigate a special class of polynomial maps which map basic semi-algebraic sets to basic ones. Let $\K$ denote either the real numbers $\R$ or the complex numbers $\C$. Let $G$ be a group which we fix for this article to be a finite group and assume that  $G$ acts linearly on $\K^n$. Hilbert observed that, in this case, the ring of invariant polynomials is a finitely generated $\K$-algebra, say
\[
\K[X_1, \dots, X_n]^G = \K[\pi_1, \dots, \pi_m] \subset \K[\X] := \K[X_1, \dots, X_n].
\]
This inclusion induces the \emph{Hilbert map} $\Pi: \K^n \to V_{\K}(I_{\Pi})$, where $V_{\K}(I_{\Pi})$ is the variety in $\K^m$ defined by the algebraic relations (syzygies) of the generators. 

In the algebraically closed case, i.e., when $\K = \C$, this map is surjective and affords a homeomorphism between the orbit space $\C^n / G$ and the variety $V_{\C}(I_{\Pi})$. Indeed, the image of the polynomial map corresponds to the categorical quotient 
$\K^n /\!/ G = \operatorname{Spec}(\C[\X]^G)$. However, in general, if $\K$ is not algebraically closed, the map fails to be surjective. In this case, the real categorical quotient $\K^n /\!/ G$ is, by the Tarski-Seidenberg Theorem, a \emph{semi-algebraic set} and can be written as a union of intersections of solution sets of polynomial inequalities.

Remarkably, it turns out that in this case the situation is more favorable: Even though, in general, the image of a polynomial map is not basic and obtaining explicit polynomial descriptions can be challenging, it was shown by Procesi and Schwarz \cite{procesi1985inequalities} that the image of the Hilbert map is a basic closed semi-algebraic set. Moreover, in the case of compact Lie groups, these inequalities can be obtained directly from the chosen fundamental invariants.

For a polynomial $p$, we consider the differential $dp$ defined by
\[
dp = \sum_{j=1}^n \frac{\partial p}{\partial x_j} dx_j.
\]
For finite (compact) $G$, we have a $G$-invariant inner product
$\langle \cdot , \cdot \rangle$, which, when applied to the differentials, yields
\[
\langle dp, dq \rangle = \sum_{j=1}^n \frac{\partial p}{\partial x_j} \cdot \frac{\partial q}{\partial x_j}.
\]
Since differentials of $G$-invariant polynomials are $G$-equivariant, the inner products $\langle d\pi_i, d \pi_j \rangle$ $(i, j \in \{1, \ldots, m\})$ are $G$-invariant, and hence every entry of the symmetric matrix polynomial
\[
M_\Pi = (\langle d\pi_i, d \pi_j \rangle)_{1 \le i, j \le m}
\] 
is an invariant polynomial. With a slight misuse of notation, we can thus represent it as a matrix polynomial in $\pi_1,\ldots,\pi_m$. Using this construction, Procesi and Schwarz \cite{procesi1985inequalities} have shown the following.

\begin{theorem}[Procesi and Schwarz]\label{procesischwarz}
Let $G \subseteq \mathrm{GL}_n(\K)$ be a finite group, and let
$\Pi = (\pi_1, \ldots, \pi_m)$ be fundamental invariants of $G$.
Then the orbit space is given by polynomial inequalities,
\[
\Pi(\K^n) = \left\{z \in V(I_\Pi) \subseteq \K^m ~\middle|~ M_\Pi(z) \text{ is positive semi-definite}\right\},
\] 
where $I_\Pi \subseteq \K[z_1, \ldots, z_m]$ is the ideal of relations of
$\pi_1, \ldots, \pi_m$.
\end{theorem}

This theorem has many applications in various areas, including differential geometry, dynamical systems, and mathematical physics (see, for example, \cite{field2007dynamics,dubrovin1998differential,huisman1999real}). Since the description of the semi-algebraic set, in combination with Artin's solution to Hilbert's 17th problem, gives rise to an equivariant Positivstellensatz, it can also be applied in polynomial optimization \cite{riener2013exploiting,Moustrou2023}. The goal of this article is twofold. In the first part of the paper, we aim to demonstrate that this remarkable result can be established with elementary results in real algebraic geometry in the case of finite groups. Specifically, we show in Theorem \ref{finitegroupsbasicclosed} that the fact that the set is basic can be directly obtained by combining sums of squares with basic invariant theoretic results. This follows from the well-known fact that the polynomial ring is a finite module over the invariant ring. Additionally, we provide a short proof of the description by Procesi and Schwarz. Note that after their original paper, Procesi and Schwarz also obtained a rather elementary argument for their statement for finite groups \cite{procesi1988defining}. Furthermore,  using standard arguments in invariant theory, like Luna's slice theorem, the finite case can be transferred to the compact case. Thus, any elementary proof for the finite case is essentially generalizable. However, the motivation for our proof is not only that it is elementary but it also serves as a stepping stone for more efficient descriptions of orbit spaces,  which is the question we focus on in the second part. 

Given any basic semi-algebraic set, it is natural to ask about the minimal number of inequalities needed to describe it. A famous result by Bröcker and Scheiderer \cite{broe,scheiderer} shows that any closed basic semi-algebraic set in $\K^n$ can be described by $n(n-1)/2$ inequalities. On the other hand, the description by Procesi and Schwarz is better, as it yields $m$ inequalities, where $m$ is the number of fundamental invariants, which is the dimension the orbit space is embedded into. However, as observed by Procesi and Schwarz, if the order $|G|$ of $G$ is odd, the Hilbert map is surjective and one does not need any inequality, although their construction still produces $m$ inequalities. This raises the natural question of how the number of inequalities is related to the structure of $G$. Bröcker \cite{brocker1998symmetric} answered this question completely: The number $k$ of inequalities needed to describe the orbit space generically, i.e. up to some lower dimensional set $T$, is exactly the maximal number for which $G$ contains an elementary abelian subgroup of order $2^k$. Although this completely answers the question, Bröcker's proof is not constructive. In the second part of the article, we turn to a first class of non-trivial examples, where we constructively build a description with the least number of inequalities as predicted by Bröcker's theorem. Furthermore, we answer a question raised by Bröcker: We give an example for which the generic description obtained from Bröcker is not a complete description of the orbit space, i.e. an example where one needs to add some lower dimension set $T$.

This article is structured as follows: In the following section, we provide elementary arguments to establish that the orbit space is a basic semi-algebraic set and give a new proof for the description due to Procesi and Schwarz, which is obtained by going through subgroup chains. The third section provides the construction of orbit spaces with the least number of inequalities, and we conclude with some open questions.

\section{Real orbit spaces for finite groups}
We now aim to provide firstly a constructive argument for the remarkable fact that $\Pi(\R^n)$ is a basic semi-algebraic set and then secondly give an elementary proof for Theorem \ref{procesischwarz}. Throughout this section, $G$ is a finite group and we use the notations of Theorem \ref{procesischwarz}.
\subsection{Notations}
Let $G$ act linearly on $\K^n$,  then we get an action on the ring of polynomials defined by \[h^\sigma=h(\sigma^{-1}(x)),\text{ for all } \sigma \in G.\] Furthermore, we will use that the polynomial ring $\K[X]$ is a $\K[X]^G$ module. Notice that $\K[\X]$ is integral over $\K[\X]^G$. Indeed, for any $f\in\K[\X]$, we have a monic  characteristic polynomial \[\chi_f(T):= \prod_{\sigma \in G}(T-f^{\sigma}) \in \K[\X]^G[T].\] Therefore, $\K[\X]$ is a finitely generated $\K[\X]^G$-module. Moreover, we can define a simple projection operator, called the \emph{Reynolds operator} which gives a projection from $\K[X]$ to $\K[X]^G$. It is defined by
\[\mathcal{R}_G:\;h\mapsto \frac{1}{|G|}\sum_{\sigma\in G} h^{\sigma}.\]

To work in the algebraic setting of ring extensions, it is  practical to identify the points in $\K^n$ with ring homomorphisms from $\K[\X]$ to $\K$. In this way we identify $V_\K(I_\Pi)$ with $\hom(\K[\X]^G,\K)$ and 
every $z\in V_\K(I_\Pi)$ is identified with the ring homomorphism $\phi_z$ defined to be the evaluation in $z$. On the other hand, since every ring homomorphism is determined uniquely by the image of $X_1,\ldots,X_n$, we can equivalently identify every $\phi\in\hom(\K[\X]^G,\K)$ with a unique point $z_\phi\in V_\K(I_\Pi)$.
\begin{proposition}
With the above  identification a point $z\in V_\K(I_\Pi)$ is in the image of $\Pi$ if and only if $\phi_z$ can be extended to a ring homomorphism from $\K[\X]$ to $\K$.
\end{proposition}

One of the basic notions in real algebraic geometry is the notion of sums of squares and 
our proofs will rely on the set of invariant sums of squares.

\begin{definition}
A polynomial $f\in\R[\X]$ is called a sum of squares if it can be decomposed into the form $f=f_1^2+\ldots +f_\ell^2$ for some polynomials $f_1,\ldots,f_\ell \in\R[\X]$. We will write $\Sigma{ \R[\X]^2}$ for the set of all these polynomials.  Furthermore, we set $\left(\Sigma\R[\X]^2\right)^G=\Sigma{ \R[\X]^2}\cap\R[X]^G$. More generally,  we call a symmetric matrix polynomial $A\in\R[\X]^{k\times k}$ a sums of squares matrix polynomial, if $A=L^tL$ for some $L\in\R[\X]^{k\times \ell}$ and we say that a sums of squares matrix polynomial is $G$ invariant if all of its entries are $G$-invariant polynomials. 

\end{definition}
Notice that the set of invariant sums of squares $\left(\Sigma\R[\X]^2\right)^G$ defines a quadratic module in the ring $\R[\X]^G$, which in general is not finitely generated (see \cite[Example 5.3]{cimprivc2009sums}). However, using the fact that $\R[X]$ is a finitely generated module it can be conveniently represented using sums of squares matrices. Indeed, let $b_1,\dots,b_l \in \R[\X]$ be generators of $\R[\X]$ over $\R[\X]^G$ and define \[B\in (\R[\X]^G)^{l\times l} \text{ by } B_{ij}:=\mathcal{R}_G(b_i b_j),\] then 
 we have the following characterization (see \cite{gatermann2004symmetry,br,Moustrou2023} for details): 

\begin{proposition}\label{prop:matrixfacto}
Let $f\in\R[X]^G$. Then $f\in \Sigma (\R[\X]^G)^2$ if and only if there exists a $G$-invariant sums of invariant squares matrix polynomial $A\in(\R[\X]^G)^{t\times t}$ with a factorization $A=L^tL$ from some $L\in(\R[\X]^G)^{k\times \ell}$ such that
\[
f=\Tr(A\cdot B)
\]
\end{proposition}
\begin{proof}
We sketch the proof for the convenience of the reader. We only consider the case when $f$ is the sum of the orbit of one square - and the general case follows directly in the same way. Let  $g\in \R[\X]$ with $f=\mathcal{R}_G(g^2)$ and write $g=\sum_{i=1}^l a_ib_i$ for some $a_1,\dots,a_l\in \R[\X]^G$. Then
\[f=\mathcal{R}_G(g^2)=\mathcal{R}_G(\sum_{i,j=1}^l a_ia_jb_ib_j)=\sum_{i,j=1}^l a_ia_j\mathcal{R}_G(b_ib_j)=a^TBa=Tr(aa^tB)\]
where $a:=(a_1,\dots,a_l)^T$.
\end{proof}

\subsection{Orbit spaces as basic semi-algebraic sets}
Based on the previous discussions it is almost directly clear that the semi-algebraic set $\Pi(\R^n)$ is basic. Indeed, it will be a consequence of the following simple observation. 
\begin{proposition}\label{prop:equivSOS}
Let $z\in V(I_\Pi)$ such that $\Pi^{-1}(z)\not\in\R^{n}$, then there exists $f\in\left(\Sigma\R[\X]^2\right)^G$ such that $\phi_z(f)<0$.
\end{proposition}
\begin{proof}
 Set $\xi=\Pi^{-1}(z)\in \C^n\setminus \R^n$, let $1\leq i\leq n$ be such that $\text{Im}(\xi_i)\neq 0$, and consider the polynomial $h:=(X_i-\text{Re}(\xi_i))^2$. Clearl, $h(\xi)<0$. Let $\mathcal{O}_G(h)$ denote the orbit of $h$ under $G$. We now construct the univariate polynomial $p\in\R[\X]^G[T]$ by 
 \[p(T)=\prod_{h'\in\mathcal{O}_G(h)} (T+h').\]
 Each of the coefficients of $p$ is in $(\Sigma\R[\X]^2)^G$.  We evaluate these coefficients in $z$ and observe that the resulting univariate polynomial will have at least one negative coefficient by  Descartes' rule of signs. Thus, for some $1\leq k\leq l=|\mathcal{O}_G(h)|$ we have an elementary symmetric polynomial  $f:=e_k(h^{\sigma_1},\dots,h^{\sigma_l}) \in\left(\Sigma\R[\X]^2\right)^G$ such that $\phi_z(f)=f(\xi)<0$.
\end{proof}

Combining this observation with  Proposition~\ref{prop:matrixfacto}, we immediately get a description of $\Pi(\R^n)$ as basic semi-algebraic set. 

\begin{theorem}\label{finitegroupsbasicclosed}
The semi-algebraic set $\Pi(\R^n)$ is basic closed, it can be represented as
\[\Pi(\K^n) = \left\{z \in V(I_\Pi) \subseteq \K^m ~\middle|~ \phi_z(B) \text{ is positive semi-definite}\right\}.\]
\end{theorem}
\begin{proof}
Clearly, the condition that $B(z)$ is positive semi-definite is necessary. Indeed, for $z\in \Pi(\R^n)$  there exists  by definition $x\in \R^n$ such that $\Pi(x)=z$, and therefore $\phi_z(B)=B(x)$. Let $v\in \R^m$. Then
$v^TBv\in (\Sigma\R[\X]^2)^G$, so $v^T\phi_z(B)v=(v^TBv)(x)\geq 0$.
On the other hand, assume that $\phi_z(B)$ is positive semi-definite but that $\phi_z$ cannot be extended. By Proposition \ref{prop:equivSOS} we know that there exists an invariant sums of squares polynomial $f$ with $\phi_z(f)<0$. Furthermore, by Proposition \ref{prop:matrixfacto} we have
\[
f=\Tr(L(x)L(x)^t B(x)).
\]
However, since $\phi_z(L(x))\phi_z(L(x)^t)$ and $\phi_z(B(x))$ are positive semi-definite matrices, we have a contradiction. 
\end{proof}

\begin{remark}
This result is quite constructive: It suffices to find generators of $\R[\X]$ over $\R[\X]^G$. If $G$ is of order $l$, $\R[\X]$ is generated by
\[ X_1^{\alpha_1}\cdots X_n^{\alpha_n} \qquad \alpha\in \{0,\dots,l-1\}^n\]
over $\R[\X]^G$ (use \cite{atiyah2018introduction} Proposition 2.16, Proposition 5.1 and Corollary 5.2).
Furthermore, in the case of  finite reflection groups these generators can be found very directly and with combinatorial methods \cite{hubert2022algorithms,debus2023reflection}. Moreover, one can also use the  $G$-harmonic polynomials, which can be obtained as the partial derivatives of the determinant of the Jacobian of the Hilbert map (see \cite{helgason1984groups} chapter III 1.1, 3.6 and 3.7).
\end{remark}

\subsection{An elementary proof for Theorem \ref{procesischwarz}}
We now want to show an elementary proof for the concrete description given in Theorem \ref{procesischwarz}. To begin, we reformulate the statement in terms of homomorphisms.

\begin{proposition}\label{equivalenceprocesi}
The following are equivalent:
\begin{itemize}
    \item [(i)] $\Pi(\R^n) = \left\{ z\in V_\R(I_\Pi) ~\middle|~ \phi_z(M_\Pi) \text{ is positive semi-definite} \right\}$.
    \item [(ii)] Let $\phi: \R[\X]^G\to \R$ be a ring homomorphism. Then $\phi$ can be extended to a ring homomorphism $\tilde{\phi}: \R[\X]\to \R$ if and only if 
    \[\phi(\langle d f,d f\rangle)\geq 0 \text{ for all  }f\in \R[\X]^G.\]
\end{itemize}
\end{proposition}
\begin{proof}
Let $z\in V_\R (I_\Pi)$. Note that $z\in\Pi(\R^n)$ if and only if $\phi_z$ can be extended to a ring homomorphism $\tilde{\phi}: \R[\X]\to \R$. Furthermore, since $\pi_1,\ldots,\pi_m$ generate the ring of invariants we have $\phi_z( M_\Pi)$ is positive semi-definite if and only if $\phi_z (\langle d f, d f\rangle ) \geq 0$ for all $f\in \R[\X]^G$.
\end{proof}

Note that $\phi_z(M_\Pi)$ being positive semi-definite in (i) is clearly necessary for $z$ being in the image of the real Hilbert map. Equivalently, $\phi (\langle d f, d f\rangle ) \geq 0$ for all $f\in \R[\X]^G$ is necessary in order to extend a homomorphism $\phi: \ \R[\X]^G\to \R$ to a homomorphism $\R[\X]\to \R$. So we will show that this is also sufficient, and our strategy consists in extending homomorphisms in steps, first to the ring of invariant polynomials for some subgroup $H$ of $G$. In order to do this we have to make sure that our positivity condition in (ii) extends also to $\R[\X]^H$ in this case. The following is in fact the core of the argument. 
\begin{proposition}\label{condition-extension}
Let $H$ be a subgroup of $G$ and let $\phi: \R[\X]^G\to \R$ be a homomorphism. Suppose \[\phi(\langle d f,d f\rangle)\geq 0 \text{ for all  }f\in \R[\X]^G,\] and $\phi$ can be extended to a homomorphism $\phi_H: \R[\X]^H\to \R$. Then
\[ \phi_H(\langle d f,d f\rangle) \geq 0 \text{ for all  }f\in \R[\X]^H.\]
\end{proposition}

\begin{proof}
The ring of invariants $\R[\X]^H$ is a finitely generated $\R$-algebra, say $\R[\X]^H=\R[p_1,\dots,p_k]$ and write $\pi_i=q_i(p_1,\dots,p_k)$ for some $q_i\in\R[Y_1,\dots,Y_k]$. Consider
\begin{align*}
\Pi: \C^n &\to V_\C(I_\Pi), \ x\mapsto (\pi_1(x),\dots,\pi_m(x)),\\
P: \C^n &\to V_\C(I_P), \ x\mapsto (p_1(x),\dots,p_k(x)) \text{ and}\\
Q: V_\C(I_P) &\to V_\C(I_\Pi), \ x\mapsto (q_1(x),\dots,q_k(x))
\end{align*}
and furthermore, define the points 
\[z:=z_\phi:=(\phi(\pi_1),\dots,\phi(\pi_m))\in V_\R(I_\Pi) \text{ and } y:=y_{\phi_H}:=(\phi_H(p_1),\dots,\phi_H(p_k))\in V_\R(I_P)\]
corresponding to $\phi$ and $\phi_H$. Since $P$ is surjective, there is $x\in \C^n$ with $P(x)=y$, so $\Pi(x)=z$. Denote the corresponding total derivatives of $\Pi$, $P$ and $Q$ in $x$, respectively in $y$, by 
\[D_Q^y:\C^m \to \C^k, ~ 
D_P^x:\C^k \to \C^n ~ \text{and} ~
D_\Pi^x: \C^m \to \C^n.\]
We can assume furthermore that the stabilizer $G_x$ of $x$ in $G$ is trivial and therefore included in $H$ (otherwise we can perturbate $y$ and therefore $x$ and $z$ a little bit by continuity). So $\dim (\im(D_P^x))=\dim(\im(D_\Pi^x))$ (see \cite{abud1983geometry}) and therefore $\im(D_P^x)=\im(D_\Pi^x)$, since $\im(D_P^x)\supseteq\im(D_\Pi^x)$ is trivial. So for every $v\in \C^k$ there is some $u\in \C^m$ with $D_\Pi^x(u)=D_P^x(v)$ and so the identity
\[ v=D_Q^x(u)+(v-D_Q^x (u))\in \im D_Q^y + \ker D_P^x \]
shows that $\im D_Q^y + \ker D_P^x = \C^k$.\\
Let now $f\in \R[\X]^H$, say $f=g(p_1,\dots,p_k)$. Then there is $u \in \ker D_P^x$ and $v\in \im D_Q^y$ with $a:=(d g)(P(x))=u+v$, i.e. $D_P^x(u)=0$ and there is $w\in \C^m$ with $v=D_Q^y(w)$. Because $a$ and $y$ are real, we can assume that $w\in \R^m$. Then
\begin{align*}
\phi_H(\langle d f,d f\rangle) &= \langle (d f)(x),(d f)(x)\rangle \\
&= \langle D_P^x(a),D_P^x(a)\rangle\\
&= \langle D_P^x(u+v),D_P^x(u+v)\rangle \\
&= \langle D_P^x(u)+D_P^x(v),D_P^x(u)+D_P^x(v)\rangle \\
&= \langle D_P^x(D_Q^y(w)),D_P^x(D_Q^y(w))\rangle \\
&= \langle D_\Pi^x(w),D_\Pi^x(w)\rangle\\
&= \phi(\langle d h,d h\rangle)\geq 0,
\end{align*}
where $h:= w_1 \pi_1 + \cdots + w_m \pi_m \in \R[\X]^G$.
\end{proof}

Now we will collect some facts about invariants of groups of order two because we will extend our homomorphism first to the ring of invariant polynomials for some order two subgroup $H$ of $G$. This order two subgroup will correspond to a complex point and its complex conjugated point.

\begin{lemma}\label{lemmaorder2}
Let $H$ be a subgroup of $G$.
\begin{itemize}
\item[(a)] $\R[\X]^H=\R[\X]^G \oplus \ker \left({\mathcal{R}_G}\vert_{\R[\X]^H}\right)$.
\item[(b)] If $|G/H|=2$, then $\ker \left({\mathcal{R}_G}\vert_{\R[\X]^H}\right)^2 \subseteq \R[\X]^G$.
\end{itemize}
\end{lemma}
\begin{proof}
\begin{itemize}
\item[(a)] Let $g\in\ker \mathcal{R}_G\cap \R[\X]^G$. Then $g=\mathcal{R}_G(g)=0$, so $\ker \mathcal{R}_G\cap \R[\X]^G=\{0\}$. Now let $f\in \R[\X]^H$. Then 
\[f=\mathcal{R}_G(f)+f-\mathcal{R}_G(f) \in \R[\X]^G \oplus \ker \left({\mathcal{R}_G}\vert_{\R[\X]^H}\right)\] 
and
\[\R[\X]^G \oplus \ker \left({\mathcal{R}_G}\vert_{\R[\X]^H}\right)\subseteq \R[\X]^H\] 
is clear.
\item[(b)] Denote $G/H=\{H,\sigma H\}$ and let $r\in \ker \left({\mathcal{R}_G}\vert_{\R[\X]^H}\right)^2$. Then
\[0=\mathcal{R}_G(r)= \frac{1}{2}(r +  \sigma \cdot r),\]
so $\sigma r=-r$ and therefore $\sigma r^2=r^2$, i.e. $r^2$ is $G$-invariant.
\end{itemize}
\end{proof}

From these elementary considerations we now can give a proof of Theorem \ref{procesischwarz}.

\begin{proof}[Proof of Theorem \ref{procesischwarz}]
We will show property (ii) in Proposition \ref{equivalenceprocesi}: "$\Rightarrow$" is clear.
\\"$\Leftarrow$" by contradiction: Let $\phi: \R[\X]^G\to \R$ be a ring homomorphism such that 
\[\phi(\langle d f,d f\rangle)\geq 0 \text{ for all  }f\in \R[\X]^G\]
and $\phi$ cannot be extended to $\R[\X]$, i.e. $z:=(\phi(\pi_1),\dots,\phi(\pi_m)) \notin \Pi(\R^n)$. There is $x\in \C^n$ such that $\Pi(x)=z$. Since $\Pi(\bar x)=z$, there is $\sigma \in G$ such that $\sigma \cdot x=\bar x$. Then $\sigma$ has even order. Now consider the subgroups $C_\sigma=\langle \sigma\rangle$ and $C_{\sigma^2}=\langle \sigma^2\rangle$ of $G$. Since 
\[f(x)=\sigma f(x)=f(\bar x)=\overline{f(x)}\]
for all $f\in \R[\X]^{C_\sigma}$, we can extend $\phi_z$ to a homomorphism $\tilde \phi:\R[\X]^{C_\sigma}\to \R$.
\\
By Proposition \ref{condition-extension}, $\tilde \phi(\langle d f,d f\rangle) \geq 0 \text{ for all  }f\in \R[\X]^{C_\sigma}$.
\\
Since $x\notin \R^n$, there is a linear $l\in \R[\X]$ with $l(x)=i$. Consider now
\[f= R_{C_{\sigma^2}}(l) - R_{C_{\sigma}}(l) \in \ker R_{C_\sigma}.\]
Note that $\langle d f, d f\rangle \in \R_{>0}$, because $f$ is linear. Since $C_{\sigma^2}$ acts trivially on $x$ and $\sigma x = \bar x$ we get that $f(x)=i$ and therefore
$\phi(f^2)=f^2(x)=-1<0$ with $f^2\in \R[\X]^{C_\sigma}$ by Lemma \ref{lemmaorder2}.
\\
Now
\[\tilde\phi(\langle d f^2, d f^2\rangle) = \tilde\phi(\langle 2f d f, 2f d f\rangle) = \tilde\phi(4f^2) \langle d f, d f\rangle < 0,\]
which is a contradiction to $\tilde\phi(\langle d g, d g\rangle)\geq 0$ for all $g\in \R[\X]^{C_\sigma}$.
\end{proof}

Theorem \ref{procesischwarz} can be extended to compact Lie groups by Luna's slice theorem, similar to what is done  in \cite{procesi1985inequalities}. We will end this section by briefly sketching   an alternative purely algebraic approach.
If the group $G$ is not finite, then $\R(\X)^G$ will not be algebraic over $\R(\X)$, i.e. there are transcendental elements $T_1,\dots,T_k\in \R(\X)$ over $\R(\X)^G$ such that $\R(\X)|\R(\X)^G[T_1,\dots,T_k]$ is algebraic and $\R(\X)^G[T_1,\dots,T_k]|\R(\X)^G$ is purely transcendental. We can assume without loss of generality, that $T_i=X_i$. Furthermore, $\R(\X)^G[X_1,\dots,X_k]=\R(\X)^H$, where $H=\text{stab}(X_1,\dots,X_k)$. Note that $H$ is finite since $\R(\X)|\R(\X)^H$ is algebraic. We can now extend $\phi$ to $\R[\X]^H$ by the following proposition.

\begin{proposition}\label{transcendentalextension}
Let $H$ be a subgroup of $G$ such that $\R(\X)^H$ is purely transcendental over $\R(\X)^G$, say $\R(\X)^H$=$\R(\X)^G[T_1,\dots,T_k]$ for some $T_1,\dots,T_k\in \R[\X]$ transcendental over $\R(\X)^G$. Furthermore, let $\phi: \R[\X]^G\to \R$ be a ring homomorphism. Then for any $t_1,\dots,t_k\in \R$, we get an extension $\phi_H: \ \R[\X]^H\to \R$ of $\phi$ by setting $\phi_H(T_i)=t_i$ and $\phi_H(f)=\phi(f)$ for all $f\in \R[\X]^G$.
\end{proposition}
\begin{proof}
The homomorphism $\phi_H$ is well defined because $T_1,\dots,T_k$ are transcendental over $\R(\X)^G$.
\end{proof}

Since we can choose $t_1,\dots,t_k \in \R$ arbitrarily, it should be possible to extend $\phi$ to $\R[\X]^H$ in such a way that the \textit{Procesi-Schwarz matrix} $M_{X_1,\dots,X_k,\Pi}$ corresponding to $H$ is positive semi-definite, although we were not able to prove this. This would reduce then to the finite case.

\begin{remark}
Alternatively to the view point of ring homomorphisms one could work also in the spectral setting, i.e. replacing $\R^n$ and $V_\R(I_\Pi)$ by the real spectrum of $\R[\X]$ and $\R[\X]^G$. In this setup one considers extensions of orders instead of extensions of homomorphisms. This approaches should be essentially equivalent and we decided not to take the spectral point of view here in order to keep our results more accessible.
\end{remark}

\section{Describing orbit spaces with few inequalities}
In the previous section, we obtained inequalities describing the orbit space as a basic closed semi-algebraic set. The aim of this section is to find descriptions of the orbit space involving fewer inequalities. To this end, we fix again a finite group $G$ and introduce some notation.

\begin{definition}
    We say that the orbit space $\Pi(\R^n)$ (or $\R^n/G$) is \emph{described by $f_1,\dots,f_k$}$\in\R[\X]^G$ if
    \[ \Pi(\R^n)= \left\{ z\in V(I_\Pi) ~\middle|~ \phi_z(f_1) \geq 0,\dots,\phi_z(f_k)\geq 0 \right\}. \]
    Furthermore, we will say that $\Pi(\R^n)$ is \emph{generically described by $f_1,\dots,f_k$} if
    \[ \Pi(\R^n)= \left\{ z\in V(I_\Pi) ~\middle|~ \phi_z(f_1) > 0,\dots,\phi_z(f_k)> 0 \right\}\cup T,\]
    where $\dim(T) < \dim(\Pi(\R^n))$.
\end{definition}

It was already observed that if $G$ has odd order, then one needs no inequalities to describe the orbit space $\R^n/G$, i.e. the Hilbert map is surjective. More generally, Bröcker (\cite{brocker1998symmetric}[Proposition 5.6.]) proved the following about the number of inequalities needed to generically describe the orbit space.

\begin{theorem}[Bröcker]\label{bröcker}
    Let $k$ be the maximal number for which $G$ contains an elementary abelian subgroup of order $2^k$. Then the orbit space $\Pi(\R^n)$ is generically described by some $f_1,\dots,f_k\in \R[\X]^G$.
\end{theorem}

Note that the proof of Bröcker's Theorem is non-constructive. The goal of this section is to construct these few inequalities for the case where $k=1$ and for abelian groups. We answer first an open question raised by Bröcker \cite{brocker1998symmetric} and point out a small mistake in his Theorem: Bröcker writes that one needs in all examples he knows no $T$. He also states Theorem \ref{bröcker} with non-strict inequalities. The following example shows that sometimes some lower-dimension $T$ is needed and that the strict inequalities in Theorem \ref{bröcker} can not be replaced by non-strict inequalities.
\begin{example}\label{ex:lower-dimensional}
Consider $C_4=\left\langle (1,2,3,4) \right\rangle$ acting by permutation on $\R^4$. Denote by $\pi_1,\dots,\pi_k$ the fundamental invariants and by 
\[\Pi: ~ \C^4 \to \C^k\]
the Hilbert map as usual. Suppose $\Pi(\R^4)$ is described by one inequality $f\geq 0$, where $f\in \R[X_1,\dots,X_4]^{C_4}$. 
Since
\[ g(a,b,\bar a,\bar b)=(13)(24)g(a,b,\bar a,\bar b)=g(\bar a,\bar b,a,b)=\overline{g(a,b,\bar a,\bar b)}  \]
and
\[ g(a,\bar a,a,\bar a)=(1234)g(a,\bar a,a,\bar a)=g(\bar a,a,\bar a,a)=\overline{g(a,\bar a,a,\bar a)}  \]
for all $a,b\in \C$ and for all $g\in \R[X_1,\dots,X_4]^{C_4}$, we get that
\[\Pi(a,b,\bar a,\bar b)\in \R^k \quad \text{and} \quad \Pi(a,\bar a,a,\bar a)\in \R^k.\]
for all $a,b\in \C$. Therefore
\[f(a,b,\bar a,\bar b)\leq 0\]
for all $a,b,\in \C$, where the inequality is strict for $a$ or $b$ are not real. So we get that
\[f(a,b,a,b)=0\]
for all $a,b\in \R$. This means that the $2$-variate polynomial
\[g:=f(X_1,X_2,X_1,X_2)\]
vanishes on $\R^2$, so $g=0$ and thus $f(i,-i,i,-i)=g(i,-i)=0$, 
contradicting $(i,-i,i,-i)\notin \Pi(\R^4)$.
\\
In a similar way one can show that the strict inequalities in Theorem \ref{bröcker} can not be replaced by non-strict inequalities. Assume
\[\Pi(\R^n)= \left\{ z\in V(I_\Pi) ~\middle|~ \phi_z(f)\geq 0\right\}\cup T,\]
where $\dim(T) < \dim(\Pi(\R^n))$ for some $G$-invariant $f$. 
We get again
\[f(a,b,\bar a,\bar b)\leq 0\]
for all $a,b\in \C$. If
\[f(a,b,a,b)=0\]
for all $a,b\in\R$, then we get the same contradiction as before. So
\[f(a,b,a,b)<0\]
for some $a,b\in \R$. So
\[f<0 \text{ on some neighbourhood of $(a,b,a,b)$ in $\R^4$,}\]
which is a contradiction to $\dim(T)<\dim(\Pi(\R^n))$.
\end{example}

\subsection{Strategy for constructing few inequalities}
In this subsection, we introduce a strategy and some tools to prove a constructive version of Bröcker's Theorem. This is easy if $G$ is already an elementary abelian $2$-subgroup. Therefore, one might be tempted to try the following approach: 
\begin{enumerate}
    \item First we try to find an elementary abelian $2$-subgroup $H$ that corresponds to all complex conjugations that might appear. If such a subgroup exists, then we can extend every homomorphism $\phi: \R[\X]^G\to \R$ to a homomorphism $\tilde \phi: \R[\X]^H\to \R$ (see Lemma \ref{lem:broad=extendable}).
    \item Now one gets $H$-invariant polynomials $h_1,\dots,h_k$ describing the orbit space $\R^n/H$.
    \item Then one tries to symmetrize these polynomials. If one finds $G$-invariant polynomials $g_1,\dots,g_k$ describing the orbit space $\R^n/H$, then these will also generically describe the orbit space $\R^n/G$.
\end{enumerate}

The first step is in particular fulfilled if $G$ contains a so-called broad subgroup. In general, the first step fails, as shown in Example \ref{counterexample:broad}.

\begin{definition}
An elementary abelian $2$-subgroup $H$ of $G$ is called \emph{broad}, if every involution of $G$ is conjugated to an element of $H$.
\end{definition}

\begin{lemma}\label{lem:broad=extendable}
Let $H$ be a broad subgroup of $G$. Then every homomorphism
$\phi: \ \R[\X]^G\to \R$ of principal orbit type can be extended to a homomorphism $\tilde \phi:\ \R[\X]^H\to \R$.
\end{lemma}
\begin{proof}
Let $\varphi: \ \R[\X]^G\to \R$ be a homomorphism. There is an extension $\psi: \ \R[\X] \to \C$ of $\varphi$. Then $\overline{\psi}$ is also an extension of $\varphi$, i.e. there is $\sigma\in G$ such that
\[\sigma \psi = \overline{\psi}.\] 
Now $\sigma^2 \psi=\psi$ and therefore $\sigma$ is an involution, since $\psi$ has principal orbit type. Since $H$ is broad, there is $h\in H$ and $g\in G$ such that $hg=g\sigma$. Consider now the extension $\tilde \psi:=g\psi: \ \R[\X]\to \C$ of $\varphi$. Then
\[
h \tilde \psi = hg \psi = g\sigma \psi = g \overline{\psi} = \overline{\tilde \psi}
\]
and therefore 
\[\tilde \psi(f)=\tilde \psi(h f)=h \tilde \psi(f)=\overline{\tilde \psi}(f)\]
for all $f \in \R[\X]^H$. So $\tilde \varphi|_{\R[\X]^H}: ~ \R[\X]^H\to \R$ is a real extension of $\varphi$ to $\R[\X]^H$.
\end{proof}

There are at least three types of groups that contain broad subgroups. The first two results are given by \cite{guralnick2022commuting}.
\begin{theorem}[Guralnick and Robinson]
Let $G$ be quasi-simple. Then $G$ contains a broad subgroup.
\end{theorem}

\begin{remark}\label{lem:sn_broad_subgroup}
As mentioned in \cite{guralnick2022commuting}, $S_n$ contains also a broad subgroup. Indeed
\[H:=\left\langle (1,2),(3,4),\dots,\left( 2\floor*{\frac{n}{2}}-1,2\floor*{\frac{n}{2}} \right) \right\rangle\]
is a broad subgroup of $S_n$.
\end{remark}

\begin{lemma}\label{lem:conjugated_or_commute}
    Assume all maximal elementary abelian 2-subgroups of $G$ are of order $2$. Then every involution of $G$ generates a broad subgroup of $G$.
\end{lemma}
\begin{proof}
    This follows from the fact that two involutions are either conjugated or commute with a third involution: Let $\sigma, \tau \in G$ be involutions. Then
    \[(\sigma \tau)^k=1\]
    for some $k \in \N$. If $k$ is even, then $\sigma$ and $\tau$ commute with the involution $(\sigma \tau)^{\frac{k}{2}}$. If $k$ is odd then
    \[(\sigma \tau)^{\frac{k-1}{2}} \sigma \left ((\sigma \tau)^{\frac{k-1}{2}}\right)^{-1} = \tau.\]
\end{proof}

In general $G$ does not always contain a broad subgroup.

\begin{example}\label{counterexample:broad}
Consider the dihedral group 
\[D_4:=\left\langle \begin{pmatrix}
0 & 1\\
-1 & 0
\end{pmatrix}, \begin{pmatrix}
0 & 1\\
1 & 0
\end{pmatrix} \right\rangle\]
with corresponding invariant ring \[\R[X_1,X_2]^{D_4}=\R[\underbrace{X_1^2 + X_2^2}_{\pi_1}, \underbrace{X_1^4 + X_2^4}_{\pi_2}].\]
It is easy to check that $D_4$ has no broad subgroup. Furthermore, there is no elementary abelian $2$-subgroup $H$ such that every homomorphism of principal orbit type 
$\phi: \ \R[\X]^{D_4}\to \R$ can be extended to a homomorphism $\tilde \phi:\ \R[\X]^H\to \R$. To show this it suffices to show this for all maximal elementary abelian $2$-subgroups:
\begin{itemize}
    \item [case 1:] \[H:=\left\langle \begin{pmatrix}
0 & 1\\
1 & 0
\end{pmatrix}, \begin{pmatrix}
-1 & 0\\
0 & -1
\end{pmatrix} \right\rangle\]
Consider the homomorphism \[\phi: \ \R[\X]^{D_4}\to \R, g(\pi_1,\pi_2)\mapsto g(0,2).\]
Now $f:=X_1X_2\in\R[X_1,X_2]^H$ with $f^2\in\R[X_1,X_2]^G$ and
\[\phi(f^2)=\phi\left(\frac{\pi_1^2- \pi_2}{2}\right)=-1<0,\]
so $\phi$ cannot be extended to a homomorphism $\tilde \phi:\ \R[\X]^H\to \R$. Furthermore, it is easy to check that $\phi$ has principal orbit type, since it corresponds to the complex preimage $(1,i)$.
\item [case 2:] \[H:=\left\langle \begin{pmatrix}
1 & 0\\
0 & -1
\end{pmatrix}, \begin{pmatrix}
-1 & 0\\
0 & 1
\end{pmatrix} \right\rangle\]
Consider the homomorphism \[\phi: \ \R[\X]^{D_4}\to \R, g(\pi_1,\pi_2)\mapsto g(0,-8).\]
Now $f:=X_1^2-X_2^2\in\R[X_1,X_2]^H$ with $f^2\in\R[X_1,X_2]^G$ and
\[\phi(f^2)=\phi\left(2\pi_2- \pi_1^2\right)=-16<0,\]
so $\phi$ can not be extended to a homomorphism $\tilde \phi:\ \R[\X]^H\to \R$. Furthermore, it is easy to check, that $\phi$ has principal orbit type, since it corresponds to the complex preimage $(1+i,1-i)$.
\end{itemize}
\end{example}

\subsection{Bröcker's Theorem for $k=1$}
In this subsection, we will show how to describe the orbit space by one inequality if $k=1$ in Bröcker's Theorem \ref{bröcker}. To this end, we need to understand first the orbit space for groups of order two.

\begin{lemma}\label{lem:extendingtoZ2}
Let $G$ be a group of order two. The orbit space $\Pi(\R^n)$ is described by
\[f=\sum_{i=1}^n f_i^2\]
where $f_i:=X_i-\mathcal{R}_G(X_i)$.
\end{lemma}
\begin{proof}
It suffices to show that a homomorphism $\varphi: ~ \R[\X]^G\to \R$ can be extended to a homomorphism $\tilde \varphi: ~ \R[\X]\to \R$ if and only if $\varphi(f)\geq 0$.
First note that $f$ is $G$-invariant by Lemma \ref{lemmaorder2}, so $\varphi(f)\in \R$ and one implication is clear because $f$ is a sum of squares. For the other implication, let $\varphi: ~ \R[\X]^G\to \R$ be a homomorphism such that $\varphi(f)\geq 0$. The homomorphism $\varphi$ can be extended to a homomorphism $\tilde \varphi: ~ \R[\X]\to \C$. Note that
\[\R[\X]=\R[\X]^G[f_1,\dots,f_n]\]
\underline{Case 1:} $\varphi(f_i^2)=0$ for all $i\in \{1,\dots,n\}$. Then $\tilde\varphi(f_i)=0$ for all $i\in \{1,\dots,n\}$ and therefore $\im \tilde \varphi\subseteq \R$.
\\
\underline{Case 2:} $\varphi(f_i^2)>0$ for some $i\in \{1,\dots,n\}$. Then $\tilde\varphi(f_i)\in \R\setminus \{0\}$. Now $\tilde \varphi(f_j)\in \R$ for all $j\in \{1,\dots,n\}$ since $\tilde \varphi(f_i) \tilde \varphi(f_j)=\varphi(f_if_j)\in \R$ by Lemma \ref{lemmaorder2}. So again $\im \tilde \varphi\subseteq \R$.
\end{proof}

\begin{theorem}\label{case_k=1}
Let $G$ be a group such that all maximal elementary abelian 2-subgroups of $G$ are of order $2$. Then $\Pi(\R^n)$ is generically described by some $g\in \R[\X]^G$. \\
More precisely, if $|G|=q2^l$ with $q$ odd and 
\[H_1\subseteq H_2\subseteq \dots \subseteq H_l\subseteq G\]
are nested subgroups with $|H_i|=2^i$, then one can choose
\[g= \prod_{\sigma H_l\in G/H_l}\left( \sigma R_{H_l}\left(\sum_{i=1}^n\big(X_i-R_{H_1}\left(X_i\right) \big)^2\right) \right).\] 
\end{theorem}
\begin{proof}
It suffices to show that there is some $g\in \R[\X]^G$, such that for every homomorphism $\varphi:~ \R[\X]^G\to \R$ of principal orbit type, that $\varphi(g)\geq 0$ if and only if $\varphi$ can extended to a homomorphism $\R[\X]\to \R$.
\\
By Lemma \ref{lem:conjugated_or_commute}, $H_1$ is a broad subgroup of $G$. So by Lemma \ref{lem:broad=extendable} every homomorphism of principal orbit type
$\varphi: \ \R[\X]^G\to \R$ can be extended to a homomorphism $\varphi_H:\ \R[\X]^{H_1}\to \R$.
\\
Since $H_1$ is of order $2$, there is a $H_1$-invariant sum of squares 
\[f_1=\sum_{i=1}^n\big(X_i-R_{H_1}\left(X_i\right) \big)^2\] 
such that $\varphi_H$ can be extended to a homomorphism $\tilde \varphi: \ \R[\X]\to \R$ if and only if $\varphi_H(f_1)\geq 0$ by Lemma \ref{lem:extendingtoZ2}.
The argument now follows from the following two lemmas.

\begin{lemma}\label{claim1}
For every $1\leq i \leq l$ there exists a sum of squares $f_i\in \R[\X]^{H_i}$ such that
$\varphi$ can be extended to a homomorphism $\tilde \varphi: \ \R[\X]\to \R$ if and only if $\varphi_H(f_i)\geq 0$. 
\end{lemma}
\begin{proof} We go by  induction on $i$. The base case $i=1$ is done above. For the induction step from $i$ to $i+1$ we suppose the claim holds for some $1\leq i\leq l-1$, i.e. there is $f_i\in \R[\X]^{H_i}$ such that $\varphi$ can be extended if and only if $\varphi_H(f_i)\geq 0$. Either $f_i$ is already $H_{i+1}$-invariant and we are done or
\[\R(\X)^{H_i}=\R(\X)^{H_{i+1}}(f_i)=\R(\X)^{H_{i+1}}(\sigma f_i).\]
Consider the sum of squares $f_{i+1}:=f_i+\sigma f_i\in \R[\X]^H_{i+1}$, where $H_{i+1}/H_{i}$ is generated by $\sigma\in G$. If $\varphi$ can be extended to a homomorphism $\tilde \varphi: \ \R[\X]\to \R$, then $\varphi_H(f_{i+1})\geq 0$, since $\tilde \varphi$ is real and $f_{i+1}$ is a sum of squares.

If $\varphi_H(f_{i+1})\geq 0$, then $\varphi_H(f_i)\geq 0$ or $\varphi_H(\sigma f_i)\geq 0$, because $\sigma f_i$ is $H_i$-invariant. In the first case we are done by the induction hypothesis and in the second case we consider the homomorphism $\sigma ^{-1} \varphi_H$ instead of $\varphi_H$: The homomorphism $\sigma ^{-1} \varphi_H$ is
a real extension of $\varphi$ and $\sigma^{-1}\varphi_H(f_i)\geq 0$, so we can use the induction hypothesis.\end{proof}
By Lemma \ref{claim1} there is a $H_l$-invariant sum of squares $f_l$ such that $\varphi$ can be extended to a homomorphism $\tilde \varphi: \ \R[\X]\to \R$ if and only if $\varphi_H(f_l)\geq 0$.

\begin{lemma}\label{case2} $\varphi$ can be extended to a homomorphism $\tilde \varphi: \ \R[\X]\to \R$ if and only if 
\[\varphi(\prod_{\sigma H_l \in G/H_l} \sigma f_l)\geq 0.\]
\end{lemma}
\begin{proof}
The if case is clear, since $f:=\prod_{\sigma H_l \in G/H_l} \sigma f_l$ is a sum of squares. For the other direction assume that $\varphi(f)\geq 0$ and consider the complex extension $\tilde \varphi: \ \R[\X]\to \C$ of $\varphi$. Now
\[0\leq \varphi(f)=\prod_{\sigma H_l \in G/H_l} \tilde \varphi(\sigma f_l)=\prod_{\sigma H_l \in G/H_l}  \sigma\tilde\varphi|_{\R[\X]^H}(f_l).\]
If the image of $\sigma\tilde\varphi|_{\R[\X]^H}$ is not real for some $\sigma$, then also its complex conjugated appears in the product, so their product is positive. Since $|G/H_l|$ is odd, there has to be $\sigma H_l\in G/H_l$, such that $\sigma \tilde \varphi(f_l)\geq 0$ and $\sigma \tilde \varphi|_{\R[\X]^H}$ is a real homomorphism. Now we apply Claim 1 to $\sigma \varphi|_H$ instead of $\varphi|_H$.
\end{proof}
The poof now follows directly from Lemmas \ref{claim1} and \ref{case2}.
\end{proof}

Note that Theorem \ref{case_k=1} is constructive. We give some examples.
\begin{example}
\begin{enumerate}
\item Consider $S_3$ with the standard action on $\R[X_1,X_2,X_3]$. Following the proof of Theorem \ref{case_k=1} we consider first some order two subgroup of $S_3$, say $S_2=\langle (12) \rangle$. We obtain using Lemma \ref{lem:extendingtoZ2} the $S_2$-invariant polynomial 
\[f= (X_1-R_{S_2}(X_1))^2+(X_2-R_{S_2}(X_2))^2+(X_3-R_{S_2}(X_3))^2 =(X_1-X_2)^2 \]
in Claim 1. This polynomial is symmetrized in Claim 2 and we get, that the orbit space $\R^3/S_3$ is generically described by the discriminant 
\[g=\prod_{i=0}^2 (123)^i f = \prod_{1\leq i<j\leq 3}(X_i-X_j)^2.\]
\item Consider the Quaternion group
\[Q_8=\left\langle 
\begin{pmatrix}
0 &-1 & 0 & 0\\
1 & 0 & 0 & 0\\
0 & 0 & 0 &-1\\
0 & 0 & 1 & 0
\end{pmatrix},
\begin{pmatrix}
0 & 0 & -1 & 0\\
0 & 0 &  0 & 1\\
1 & 0 &  0 & 0\\
0 &-1 &  0 & 0
\end{pmatrix},
\begin{pmatrix}
0 & 0 &  0 & -1\\
0 & 0 & -1 &  0\\
0 & 1 &  0 &  0\\
1 & 0 &  0 &  0
\end{pmatrix}\right\rangle\]
with the canonical action on $\R[X_1,X_2,X_3,X_4]$. Consider again first the unique order two subgroup $H=\langle -I_4 \rangle$. In Theorem \ref{case_k=1} Claim 1 we obtain the $H$-invariant polynomial
\[f=X_1^2+X_2^2+X_3^2+X_4^2\]
using Lemma \ref{lem:extendingtoZ2}. Since $f$ is already $Q_8$-invariant, we don't need to apply the second step. So the orbit space $\R^4/Q_8$ is already generically described by $f$.
\item Consider the Dihedral group 
\[D_5=\left\langle 
\begin{pmatrix}
0 & 0 & 0 & 0 & 1\\
1 & 0 & 0 & 0 & 0\\
0 & 1 & 0 & 0 & 0\\
0 & 0 & 1 & 0 & 0\\
0 & 0 & 0 & 1 & 0
\end{pmatrix},
\begin{pmatrix}
0 & 0 & 0 & 0 & 1\\
0 & 0 & 0 & 1 & 0\\
0 & 0 & 1 & 0 & 0\\
0 & 1 & 0 & 0 & 0\\
1 & 0 & 0 & 0 & 0
\end{pmatrix}\right\rangle\]
with the canonical action on $\R[X_1,X_2,X_3,X_4,X_5]$. Theorem \ref{case_k=1} yields that the orbit space $\R^5/D_5$ is generically described by
\begin{align*}
\big((X_1-X_5)^2+&(X_2-X_4)^2\big)
\big((X_2-X_1)^2+(X_3-X_5)^2\big)
\big((X_3-X_2)^2+(X_4-X_1)^2\big) \\
\big(&(X_4-X_3)^2+(X_5-X_2)^2\big)
\big((X_5-X_4)^2+(X_1-X_3)^2\big).    
\end{align*}
\end{enumerate}
\end{example}

\subsection{Bröcker's Theorem for abelian groups}

In this subsection, we want to give a constructive version of Bröcker's Theorem \ref{bröcker} for abelian groups. To this end, we need to understand first what happens if we have a normal subgroup of $G$.

\begin{lemma}\label{directproduct}
Let $H$ be a normal subgroup of $G$. Assume that:
\begin{enumerate}
    \item There are $g_1,\dots,g_k\in \R[\X]^G$ such that every homomorphism $\varphi_G: ~ \R[\X]^G\to \R$ can be extended to a homomorphism $\varphi_H: ~ \R[\X]^H\to \R$, if and only if $\varphi_G(g_1)\geq 0,\dots,\varphi_G(g_k)\geq 0$.
    \item There is $h\in \R[\X]^H$ such that every homomorphism $\varphi_H: ~ \R[\X]^H\to \R$ can be extended to a homomorphism $\varphi: ~ \R[\X]\to \R$, if and only if $\varphi_H(h)\geq 0$.
\end{enumerate}
Then a homomorphism $\varphi_G: ~ \R[\X]^G\to \R$ can be extended to a homomorphism $\varphi: ~ \R[\X]\to \R$, if and only if $\varphi_G(\mathcal{R}_G(h))\geq 0,\varphi_G(g_1)\geq 0,\dots,\varphi_G(g_k)\geq 0$.
\end{lemma}
\begin{proof}
Let $\varphi_G: ~ \R[\X]^G\to \R$ be a homomorphism. 
\\
"$\Rightarrow$:" If $\varphi_G$ can be extended to a homomorphism $\varphi: ~ \R[\X]\to \R$, then $\varphi_G(g_1)\geq 0,\dots,\varphi_G(g_k)\geq 0$ by (1). For every $\sigma\in G$ we have that $\sigma \varphi|_{\R[\X]^H}$ is an extension of $\varphi_G$, so $\varphi(\sigma h)=\sigma \varphi|_{\R[\X]^H}(h)\geq 0$ by (2) and therefore $\varphi_G(\mathcal{R}_G(h)))\geq 0$.
\\
"$\Leftarrow$:" Suppose $\varphi_G(\mathcal{R}_G(h))\geq 0,\varphi_G(g_1)\geq 0,\dots,\varphi_G(g_k)\geq 0$. Then $\varphi_G$ can be extended to a homomorphism $\varphi_H: ~ \R[\X]^H\to \R$ by (1). Now $\sigma h$ is $H$-invariant for all $\sigma \in G$, since $H$ is a normal subgroup of $G$. Now
\[\sum_{\sigma \in G} \varphi_H(\sigma h)=|G|\varphi_H(\mathcal{R}_G(h)) =|G|\varphi_G(\mathcal{R}_G(h)) \geq 0\]
and since all the terms are real, $\varphi_H(\sigma_0 h)\geq 0$ for some $\sigma_0 \in G$. 
Then $\sigma_0 \varphi_H(h)= \varphi_H(\sigma_0 h)\geq 0$ and $\sigma_0 \varphi_H\in \hom(\R[\X]^H, \R)$, so we can extend $\sigma_0 \varphi_H$ to a homomorphism $\varphi: ~ \R[\X]\to \R$ by (2).
\end{proof}

Note that Theorem \ref{case_k=1} includes in particular cyclic groups of even order. The following example shows that Lemma \ref{lem:extendingtoZ2} together with Lemma \ref{directproduct} give a polynomial of lower degree describing generically the orbit space in this case. Moreover, this example generalizes Example \ref{ex:lower-dimensional} above.

\begin{example}\label{cyclic_group}
Let $G=C_m$ be a cyclic group. If $m$ is odd, the Hilbert map is surjective, so we need no inequality. If $m$ is even, write $m=2^k q$, where $q$ is odd. Furthermore, consider the cyclic subgroups
\[\{\id \} = C_1 \subset C_2\subset C_4 \subset \dots \subset C_{2^k} \subset C_m\]
and the $C_m$-invariant polynomials
\[f_i:= R_{C_m}\left( \sum_{j=1}^n \left(R_{C_{2^{i-1}}}(X_j) - R_{C_{2^i}}(X_j) \right)^2\right)\]
Then from Lemma \ref{lem:extendingtoZ2} and Lemma \ref{directproduct} we obtain that the orbit space $\R^n/C_m$ is generically described by $f_1$. Furthermore, $\R^n/C_m$ is described by $f_1,\dots,f_k$: 
For every non-real preimage $x\in \C^n\setminus \R^n$, there is $\sigma$ of order $2^i$ with $\sigma x =\bar x$ (replace $\sigma$ by $\sigma ^r$ for some odd $r$). Now it is easy to check that $f_i(x)<0$:
\begin{align*}
f_i(x)&=R_{C_m}\left( \sum_{j=1}^n \left(R_{C_{2^{i-1}}}(x_j) - R_{C_{2^i}}(x_j) \right)^2\right)\\
&=R_{C_m}\left( \sum_{j=1}^n \left(\frac{1}{|C_{2^{i-1}}|}\sum_{l=1}^{2^{i-1}} \sigma^{2l}(x_j) - \frac{1}{|C_{2^{i}}|}\sum_{l=1}^{2^{i}} \sigma^{l}(x_j) \right)^2\right)\\
&=\frac{1}{|C_{2^{i}}|}R_{C_m}\left( \sum_{j=1}^n \left(2\sum_{l=1}^{2^{i-1}} \sigma^{2l}(x_j) - \sum_{l=1}^{2^{i-1}} \sigma^{2l}(x_j)+\sigma^{2l+1}(x_j) \right)^2\right)\\
&=\frac{1}{2^{i}}R_{C_m}\left( \sum_{j=1}^n \left(\sum_{l=1}^{2^{i-1}} \sigma^{2l}(x_j) - \sigma^{2l+1}(x_j) \right)^2\right)\\
&=\frac{1}{2^{i}m} \sum_{\tau\in C_m} \tau\left(  \sum_{j=1}^n \left(\sum_{l=1}^{2^{i-1}} \sigma^{2l}(x_j-\overline{x_j}) \right)^2\right)\\
&=\frac{1}{2m} \sum_{\tau\in C_m} \left(  \sum_{j=1}^n \left(\tau x_j - \overline{\tau x_j} \right)^2\right)<0.
\end{align*}
In general, $\R^n/C_m$ can not be described by less than $k$ polynomials, which can be shown in a similar way as in Example \ref{ex:lower-dimensional}.
\end{example}

We can now prove Bröcker's Theorem for abelian groups.
\begin{theorem}\label{abelian}
Let $G$ be abelian and choose $k\in \N_0$ such that the maximal abelian $2$-subgroups of $G$ are of order $2^k$. Then $\R^n/G$ is generically described by $f_1,\dots,f_k \in \R[\X]^G$.
\end{theorem}
\begin{proof}
Since $G$ is abelian, it is isomorphic to the direct product of cyclic groups, say $G\cong H_1\times \dots \times H_k \times H_{k+1}\times \dots \times H_m$, where $H_1,\dots,H_k$ are of even order and $H_{k+1},\dots,H_m$ are of odd order.
\\
By applying iteratively Remark \ref{cyclic_group} and Lemma \ref{directproduct} we find $g_1,\dots,g_i\in \R[\X]^G$ ($1\leq i \leq k$) such that every homomorphism $\varphi_G: ~ \R[\X]^G\to \R$ of principal orbit type can be extended to a homomorphism $\varphi_i: ~ \R[\X]^{H_{i+1}\times \dots \times H_m}\to \R$, if and only if $\varphi_G(g_1)\geq 0,\dots,\varphi_G(g_i)\geq 0$. Since $H_{k+1}\times \dots \times H_m$ is odd, we can extend every homomorphism $\varphi_k: ~ \R[\X]^{H_{k+1}\times \dots \times H_m}\to \R$ to a homomorphism $\varphi: ~ \R[\X]\to \R$.
\end{proof}

The following example shows how to use the proof of Theorem \ref{abelian} to construct inequalities describing generically the orbit space of abelian groups.

\begin{example}\label{ex:abelian}
Consider the abelian group
\[G:=\left\langle \begin{pmatrix}
0 & 0 & 0 & 1\\
1 & 0 & 0 & 0\\
0 & 1 & 0 & 0\\
0 & 0 & 1 & 0
\end{pmatrix}\begin{pmatrix}
-1& 0 & 0 & 0\\
0 &-1 & 0 & 0\\
0 & 0 &-1 & 0\\
0 & 0 & 0 &-1
\end{pmatrix} \right\rangle\cong C_4\times C_2\]
and the subgroup
$H:=\left\langle \begin{pmatrix}
-1& 0 & 0 & 0\\
0 &-1 & 0 & 0\\
0 & 0 &-1 & 0\\
0 & 0 & 0 &-1
\end{pmatrix} \right\rangle\cong C_2$
of $H$. The orbit space $\R^4/H$ is generically described by the polynomial 
\[h=X_1^2+X_2^2+X_3^2+X_4^2\]
by Example \ref{cyclic_group}. Furthermore, we can also apply Example \ref{cyclic_group} to the ring extension 
\[\R[\X]^G\subseteq \R[\X]^H\] and get that a homomorphism $\phi: ~\R[\X]^G\to \R$ of principal orbit type can be extended to a homomorphism 
$\tilde \phi: ~ \R[\X]^H\to \R$ if and only if $\phi(g)\geq 0$, where
\[g=(X_1^2-X_3^2)^2+(X_2^2-X_4^2)^2+(X_1X_2-X_3X_4)^2+(X_1X_4-X_2X_3)^2.\] Now by Lemma \ref{directproduct} the orbit space $\R^4/G$ is generically described by the polynomials $g$ and $\\mathcal{R}_G(h)=h$.
\\
By considering first 
\[H_2=\left\langle \begin{pmatrix}
0 & 0 & 0 & 1\\
1 & 0 & 0 & 0\\
0 & 1 & 0 & 0\\
0 & 0 & 1 & 0
\end{pmatrix} \right\rangle\]
instead of $H$, one obtains that $\R^4/G$ is also generically described by
\[h_2=(X_1-X_3)^2+(X_2-X_4)^2 \text{ and } g_2=(\sum_{i=1}^4 X_i)^2+(\sum_{i=1}^4 X_i^3)^2+(X_1X_2^2+X_2X_3^2+X_3X_4^2+X_4X_1^2)^2.\]
\end{example}

\section{Conclusion and open questions}
We conclude our article with some open questions and points for further inquiry. We were able to show Bröcker's result constructively for some classes of groups. However, our techniques developed here do not seem to directly apply to other interesting groups. For example, note that the symmetric $S_n$ contains the broad subgroup
\[H_n:=\left\langle (1,2),(3,4),\dots,\left( 2\floor*{\frac{n}{2}}-1,2\floor*{\frac{n}{2}} \right) \right\rangle\] 
for all $n$ by Remark \ref{lem:sn_broad_subgroup}. Still, our techniques fail to give a generic description of the orbit space $\R^n/S_n$ with $\lfloor \frac{n}{2} \rfloor$ inequalities for $n\geq 4$. By Theorem \ref{abelian} we get that $\R^n/H_n$ is described by the $\floor*{\frac{n}{2}}$ polynomials
\[f_1:=(X_1-X_2)^2,f_2:=(X_3-X_4)^2,\dots,f_{\floor*{\frac{n}{2}}}:=(X_{2\floor*{\frac{n}{2}}-1}-X_{2\floor*{\frac{n}{2}}})^2\]
and since $H_n$ is broad we can extend every homomorphism in $\hom(\R[\X]^{S_n},\R)$ to a homomorphism in $\hom(\R[\X]^{H_n},\R)$ by Lemma \ref{lem:broad=extendable}. A natural approach might be to symmetrize the polynomials $f_1,\dots,f_{\floor*{\frac{n}{2}}}$  in such a way that the set they describe remains the same and we could in this way obtain a description of $\R^n/H_n$ in terms of $S_n$-invariant polynomials. Then, these polynomials would also describe $\R^n/S_n$. However, we currently are not able to produce such a symmetrization.  Scheiderer mentioned to us a way to generically describe $\R^4/S_4$ with two inequalities if $S_4$ acts by permutation. We include this example with his permission. 

\begin{example}
The invariant ring $\R[\X]^{S_4}$ is generated by the elementary symmetric polynomials $e_1,e_2,e_3,e_4$. Denote the corresponding Hilbert-map by $E$. Then $z\in E(\R^4)$, if and only if the univariate polynomial
\[f=T^4-z_1T^3+z_2T^2-z_3T+z_4\]
has only real roots. Furthermore, it is well known that the nature of the roots of $f_z$ is given by the eigenvalues of its Hermite matrix $H_f$. The number of real roots of $f_z$ is equal to the signature of $H_f$, which is generically equal to the number of sign changes in the series $p_1,p_2,p_3,p_4$ of the leading principal minors of $H_f$. So the signature of $H_f$ is non-negative. Now, since $p_1=4$ is constant, $H_f$ is positive definite, if and only if $p_2 p_4>0$ and $p_3>0$. I.e., the orbit space $\R^4/S_4$ is generically described by $p_2p_4$ and $p_3$.
\end{example}
This example  can be generalized to get $n-2$ inequalities that generically describe the orbit space $\R^n/S_n$, but it seems hopeful that a similar approach might give the expected number of $\lfloor \frac{n}{2}\rfloor$ inequalities. This setup seems particularly interesting, since the orbit space $\R^n/S_n$ can be viewed as the space of hyperbolic univariate polynomials. 

A further question consists in the relationship of our sums of squares approach in section 2 and the general description. In \cite[Theorem 4.9]{procesi1985inequalities}  a description of the orbit space using covariants is given and the description we obtain by Theorem \ref{finitegroupsbasicclosed} naturally contains this description, - as covariants are a natural subset of the ring extension we used to represent the sums of squares. However, our description additionally gives redundant inequalities and it would be interesting to explore if there is a simple way to directly remove these redundant inequalities to arrive at a more compact description. 

Finally, the description of the orbit space in terms of the matrix polynomial $M_\Pi$ gives a natural finitely generated quadratic module in the invariant ring:
\[
\left\{\Tr\left(A\cdot M_\Pi\right) ~\middle|~ A \text{ is a sums of invariant squares matrix polynomial}\right\}\subseteq \R[\X]^G
\]
This quadratic module naturally is contained in the invariant sums of squares polynomial, but might be substantially smaller. It is unclear how big the difference is, and this is a question which might be interesting, in particular, from the viewpoint of applications.

\section*{Acknowledgments}
This work has been supported by the Tromsø Research Foundation (grant agreement 17matteCR). The authors would like to express their gratitude to Igor Klep,  Claus Scheiderer, and Markus Schweighofer for fruitful discussions.

\printbibliography
\end{document}